\documentclass[12pt]{article}

\usepackage{amsmath,amsthm,amssymb}
\usepackage[top=30truemm,bottom=30truemm,left=25truemm,right=25truemm]{geometry}
\usepackage[dvips]{graphicx,color,psfrag}
\usepackage{bm}

\theoremstyle{plain}
\newtheorem{definition}{Definition}
\newtheorem{thm}[definition]{Theorem}

\newtheorem{cor}[definition]{Corollary}

\newtheorem{rem}[definition]{Remark}

\title{A necessary and sufficient condition for a prime to be \\ an integer group determinant of certain $p$-groups}
\author{Yuka Yamaguchi and Naoya Yamaguchi}
\date{\today}

\begin{document}

\maketitle

\begin{abstract}
We give a necessary and sufficient condition for a prime to be an integer group determinant for an arbitrary abelian $p$-group of the form ${\rm C}_{p} \times H$, where ${\rm C}_{p}$ is the cyclic group of order $p$. 
Also, 
we show that under certain conditions, the integer group determinant of a finite group $G$ that is prime is the integer group determinant of the abelianization of $G$. 
As a result, we know that the integer group determinant of a $p$-group that is prime is the integer group determinant of its abelianization. 
\end{abstract}

\section{Introduction}
For a finite group $G$, 
let $x_{g}$ be a variable for each $g \in G$, 
and let $\mathbb{Z}[x_{g}]$ be the multivariate polynomial ring in the $x_{g}$ over $\mathbb{Z}$. 
The group determinant $\Theta_{G}(x_{g})$ of $G$ is defined as follows: 
$$
\Theta_{G}(x_{g}) := \det{(x_{g h^{- 1}})_{g, h \in G}} \in \mathbb{Z}[x_{g}]. 
$$
When the variables are all integer-valued, 
the group determinant is called an integer group determinant. 
Let $S(G)$ denote the set of all integer group determinants of $G$; 
that is, 
\begin{align*}
S(G) := \left\{ \Theta_{G}(a_{g}) \: | \: a_{g} \in \mathbb{Z} \right\}. 
\end{align*}
Let ${\rm C}_{n} := \{ \overline{1}, \overline{2}, \ldots, \overline{n} \}$ denote the cyclic group of order $n$, 
and let $p$ be a prime. 
It holds that $\{ m \in \mathbb{Z} \mid \gcd(m, n) = 1 \} \subset S({\rm C}_{n})$ from \cite[Lemma~2]{MR624127} or \cite[Theorem~1]{MR550657}. 
Also, 
it follows from \cite[Theorem~2]{MR550657} that $p \not\in S({\rm C}_{p^{n}})$. 
From the above, we obtain the following: 
{\it For any prime $q$, 
it holds that $q \in S({\rm C}_{p^{n}})$ if and only if $q^{p - 1} \equiv 1 \pmod{p}$.} 

In this paper, 
we give a necessary and sufficient condition for a prime to be an integer group determinant for an arbitrary abelian $p$-group of the form ${\rm C}_{p} \times H$. 
\begin{thm}\label{thm:1}
Let $G := {\rm C}_{p} \times H$ be an abelian $p$-group with $|H| = p^{n - 1} \: (n \geq 2)$. 
Then for any prime $q$, it holds that $q \in S(G)$ if and only if $q^{p - 1} \equiv 1 \pmod{p^{n}}$. 
\end{thm}

We state two facts derived from Theorem~$\ref{thm:1}$. 
(i) For a finite group $G$ and an abelian subgroup $K$ of $G$, 
it holds that $S(G) \subset S(K)$ \cite[Theorem~1.4]{https://doi.org/10.48550/arxiv.2203.14420}. 
From this and the $H = {\rm C}_{p}$ case of Theorem~$\ref{thm:1}$, 
we can obtain the following theorem proved by Mahoney and Newman \cite[Theorem~3]{MR590367}: 
{\it Let $H$ be a finite abelian group and let $G := {\rm C}_{p^{l}} \times H$, 
where $p$ is a prime dividing $|H|$. 
Then for any prime $q \in S(G)$, 
we have $q^{p - 1} \equiv 1 \pmod{p^{2}}$.} 
By using this theorem, 
they \cite[Theorem~4]{MR590367} showed that for a finite abelian group $G$ of order $n$, 
it holds that $\{ m \in \mathbb{Z} \mid \gcd(m, n) = 1 \} \subset S(G)$ if and only if $G \cong {\rm C}_{n}$. 
(ii) A Wieferich prime is a prime number $p$ such that $2^{p - 1} \equiv 1 \pmod{p^{2}}$. 
From the $H = {\rm C}_{p}$ case of Theorem~$\ref{thm:1}$, 
we obtain the following corollary. 
\begin{cor}\label{cor:2}
It holds that $p$ is a Wieferich prime if and only if $2 \in S({\rm C}_{p} \times {\rm C}_{p})$. 
\end{cor}

Also, 
we show that under certain conditions, the integer group determinant of a finite group $G$ that is prime is the integer group determinant of the abelianization of $G$. 
For a finite group $G$, 
let $G' := [G : G]$ be the commutator group of $G$, 
let $S(G)_{\rm prime} := \{ m \in S(G) \mid m \: \: \text{is a prime} \}$, 
let $\widehat{G}$ be a complete set of representatives of the equivalence classes of irreducible representations of $G$ over $\mathbb{C}$, 
and let $\overline{G}$ be the set of degree one representations of $G$ over $\mathbb{C}$. 
Then we have the following theorem. 

\begin{thm}\label{thm:3}
Let $G$ be a finite group and suppose that there exists $d \geq 2$ satisfying $d \mid \deg{\varphi}$ for all $\varphi \in \widehat{G} \setminus \overline{G}$. 
If $|G|$ is odd or $d$ is even, 
then we have $S(G)_{\rm prime} \subset S(G/G')$. 
\end{thm}

Since $\deg{\varphi}$ devides $|G|$ for all $\varphi \in \widehat{G}$, 
we have the following. 

\begin{cor}\label{cor:4}
Let $G$ be a $p$-group. 
Then we have $S(G)_{\rm prime} \subset S(G/G')$. 
\end{cor}

\section{Proof of Theorem~$\ref{thm:1}$}

For a finite abelian group $G$, 
Dedekind gave the irreducible factorization of $\Theta_{G}(x_{g})$ over $\mathbb{C}$ as follows: 
$\Theta_{G}(x_{g}) = \prod_{\chi \in \widehat{G}} \sum_{g \in G} \chi(g) x_{g}$. 
This is called Dedekind's theorem. 

To prove Theorem~$\ref{thm:1}$, we use the following two corollaries. 

\begin{cor}\label{cor:5}
Let $p$ be a prime and let $G := {\rm C}_{p^{l}} \times H$ with $|H| = p^{n - l} \: (n - l \geq 1)$. 
Then we have 
$$
\Theta_{G}(a_{g}) \equiv \Theta_{H}\left( \sum_{i = 1}^{p^{l}} a_{(\overline{i}, h)} \right)^{p^{l}} \pmod{p^{n - l + 1}}. 
$$
\end{cor}
\begin{proof}
Let $H = {\rm C}_{p^{k_{1}}} \times \cdots \times {\rm C}_{p^{k_{s}}}$ and let 
$$
F(x, x_{1}, \ldots, x_{s}) := \sum_{i = 1}^{p^{l}} \sum_{i_{1} = 1}^{p^{k_{1}}} \cdots \sum_{i_{s} = 1}^{p^{k_{s}}} a_{(\overline{i}, \overline{i_{1}}, \ldots, \overline{i_{s}})} x^{i} x_{1}^{i_{1}} \cdots x_{s}^{i_{s}}. 
$$
Then from the $r = 1$ case of \cite[Theorem~2.3]{MR3882290}, 
we have 
$$
\Theta_{G}(a_{g}) \equiv \left\{ \prod_{j_{1} = 1}^{p^{k_{1}}} \cdots \prod_{j_{s} = 1}^{p^{k_{s}}} F(1, \omega_{k_{1}}^{j_{1}}, \ldots, \omega_{k_{s}}^{j_{s}}) \right\}^{\sum_{j = 0}^{l} \varphi(p^{j})} \pmod{p^{k_{1} + \cdots + k_{s} + 1}}, 
$$
where $\varphi$ is Euler's totient function and $\omega_{k} := \exp{\left( \frac{2 \pi \sqrt{- 1}}{p^{k}} \right)}$. 
Since 
$$
\prod_{j_{1} = 1}^{p^{k_{1}}} \cdots \prod_{j_{s} = 1}^{p^{k_{s}}} F(1, \omega_{k_{1}}^{j_{1}}, \ldots, \omega_{k_{s}}^{j_{s}}) = \Theta_{H}\left( \sum_{i = 1}^{p^{l}} a_{(\overline{i}, h)} \right), 
$$
the corollary is obtained. 
\end{proof}

\begin{cor}[{The $m = 1$ case of \cite[Lemma~1]{MR3522826}}]\label{cor:6}
Let $p$ be a prime, 
let $G$ be an abelian $p$-group of order $p^{n}$ and let $a \in \mathbb{Z}$ with $\gcd(a, p) = 1$. 
Then $a^{p^{n - 1}} - k p^{n} \in S(G)$ holds for any $k \in \mathbb{Z}$. 
\end{cor}

\begin{rem}\label{rem:7}
From Corollary~$\ref{cor:6}$, 
for any abelian $p$-group $G$ of order $p^{n}$, 
it holds that if $q$ is a prime with $q^{p - 1} \equiv 1 \pmod{p^{n}}$, 
then $q \in S(G)$ since there exists $k \in \mathbb{Z}$ satisfying $q = q^{p^{n - 1	}} - k p^{n}$. 
\end{rem}

\begin{proof}[Proof of Theorem~$\ref{thm:1}$]
If $q = p$, then the statement holds since $q \not\in S(G)$ (\cite[Theorem~1]{MR590367}) and $q^{p - 1} \not\equiv 1 \pmod{p^{n}}$. 
Below, we suppose that $q \neq p$. 
We can see the sufficiency from Remark~$\ref{rem:7}$. 
We prove the necessity. 
From \cite[Theorem~1.1]{MR4526227}, 
we have 
\begin{align*}
\Theta_{G}(a_{g}) = \prod_{\chi \in \widehat{{\rm C}_{p}}} \Theta_{H}\left( \sum_{i = 1}^{p} \chi(\overline{i}) a_{(\overline{i}, h)} \right). 
\end{align*}
Let $\alpha := \Theta_{H}\left( \sum_{i = 1}^{p} a_{(\overline{i}, h)} \right)$ and $\beta := \prod_{\chi \in \widehat{{\rm C}_{p}} \setminus \{ \chi_{0} \}} \Theta_{H}\left( \sum_{i = 1}^{p} \chi(\overline{i}) a_{(\overline{i}, h)} \right)$, 
where $\chi_{0}$ is the trivial character of ${\rm C}_{p}$, 
so that we can write as $\Theta_{G}(a_{g}) = \alpha \beta$. 
Suppose that there exist $a_{g} \in \mathbb{Z}$ satisfying $\Theta_{G}(a_{g}) = q$. 
Then, $\alpha = \pm q$ or $\pm 1$. 
Also, from the $l = 1$ case of Corollary~$\ref{cor:5}$, 
we have 
\begin{align*}
q = \Theta_{G}(a_{g}) \equiv \alpha^{p} \pmod{p^{n}}. 
\end{align*}
Therefore, if $p = 2$, then $q^{p - 1} \equiv 1 \pmod{p^{n}}$ holds. 
Note that if $p$ is an odd prime, 
then $\beta \geq 0$. 
To see this, observe that $\beta_{p - j} = \overline{\beta_{j}}$, 
where $\chi_{j} \colon {\rm C}_{p} \to \mathbb{C}; \overline{1} \mapsto \exp{\left( \frac{2 \pi \sqrt{- 1}}{p}j \right)}$ and $\beta_{j} := \Theta_{H}\left( \sum_{i = 1}^{p} \chi_{j}(\overline{i}) a_{(\overline{i}, h)} \right)$, and 
$\beta = \beta_{1} \beta_{2} \cdots \beta_{p - 1} = \beta_{1} \beta_{2} \cdots \beta_{\frac{p - 1}{2}} \overline{\beta_{1} \beta_{2} \cdots \beta_{\frac{p - 1}{2}}}$. 
It implies that if $p$ is an odd prime, 
then $\alpha = q$ or $1$, 
and therefore we have $q^{p - 1} \equiv 1 \pmod{p^{n}}$. 
\end{proof}

\section{Proof of Theorem~$\ref{thm:3}$}

To prove Theorem~$\ref{thm:3}$, 
we use the following remarks and Frobenius determinant theorem. 
\begin{rem}\label{rem:8}
Let $G$ be a finite group, 
let $e$ be the unit element of $G$, 
let $\mathbb{C}[x_{g}; g \in G \setminus \{ e \}]$ be the multivariate polynomial ring in the $x_{g} \: (g \in G \setminus \{ e \})$ over $\mathbb{C}$ and let 
$$
f(x_{e}) := x_{e}^{m} + a_{m - 1} x_{e}^{m - 1} + \cdots + a_{0} \: \: (a_{i} \in \mathbb{C}[x_{g}; g \in G \setminus \{ e \}]). 
$$
If $f(x_{e})^{k} \in \mathbb{Z}[x_{g}]$, 
then we have $f(x_{e}) \in \mathbb{Z}[x_{g}]$. 
\end{rem}

\begin{rem}\label{rem:9}
Let $G$ be a finite group of odd order. 
Then $- 1 \in S(G)$. 
Therefore, if $m \in S(G)$, then $- m \in S(G)$ since $S(G)$ is a monoid. 
\end{rem}

For a finite group $G$, Frobenius \cite{Frobenius1968gruppen} gave the irreducible factorization of $\Theta_{G}(x_{g})$ over $\mathbb{C}$ as follows: 
$
\Theta_{G}(x_{g}) = \prod_{\varphi \in \widehat{G}} \det{( \sum_{g \in G} x_{g} \varphi(g) )}^{\deg{\varphi}}. 
$
This is called Frobenius determinant theorem. 

\begin{proof}[Proof of Theorem~$\ref{thm:3}$]
Let $\pi \colon G \to G/G'$ be the canonical projection. 
Then we have $\overline{G} = \{ \psi \circ \pi \mid \psi \in G/G' \}$. 
From this and Frobenius determinant theorem, 
we have 
\begin{align*}
\Theta_{G}(x_{g}) 
&= \left( \prod_{\chi \in \overline{G}} \sum_{g \in G} \chi(g) x_{g} \right) \prod_{\varphi \in \widehat{G} \setminus \overline{G}} \det{\left( \sum_{g \in G} x_{g} \varphi(g) \right)}^{\deg{\varphi}} \\ 
&= \Theta_{G/G'}(x_{g G'}) \prod_{\varphi \in \widehat{G} \setminus \overline{G}} \det{\left( \sum_{g \in G} x_{g} \varphi(g) \right)}^{\deg{\varphi}}, 
\end{align*}
where $x_{g G'} := \sum_{h \in G'} x_{g h}$. 
Suppose that there exist $d \geq 2$ satisfying $d \mid \deg{\varphi}$ for all $\varphi \in \widehat{G} \setminus \overline{G}$. 
Then we have 
$$
\prod_{\varphi \in \widehat{G} \setminus \overline{G}} \det{\left( \sum_{g \in G} x_{g} \varphi(g) \right)}^{\deg{\varphi}} = \left\{ \prod_{\varphi \in \widehat{G} \setminus \overline{G}} \det{\left( \sum_{g \in G} x_{g} \varphi(g) \right)}^{k_{\varphi}} \right\}^{d}, 
$$
where $\deg{\varphi} = d k_{\varphi}$. 
Since $\Theta_{G}(x_{g})$, $\Theta_{G/G'}(x_{g G'}) \in \mathbb{Z}[x_{g}]$ hold, 
we have 
$$
\prod_{\varphi \in \widehat{G} \setminus \overline{G}} \det{\left( \sum_{g \in G} x_{g} \varphi(g) \right)}^{\deg{\varphi}} \in \mathbb{Z}[x_{g}]. 
$$
Thus, we have $\prod_{\varphi \in \widehat{G} \setminus \overline{G}} \det{\left( \sum_{g \in G} x_{g} \varphi(g) \right)}^{k_{\varphi}} \in \mathbb{Z}[x_{g}]$ from Remark~$\ref{rem:8}$. 
Suppose that $\Theta_{G}(a_{g})$ is a prime. 
Then from $d \geq 2$, 
we have $\prod_{\varphi \in \widehat{G} \setminus \overline{G}} \det{\left( \sum_{g \in G} a_{g} \varphi(g) \right)}^{k_{\varphi}} = \pm 1$. 
Therefore, if $|G|$ is odd, then we have $\Theta_{G}(a_{g}) = \pm \Theta_{G/ G'}(a_{g G'}) \in S(G / G')$ from Remark~$\ref{rem:9}$ since $| G / G' |$ is odd. 
If $d$ is even, then we have $\Theta_{G}(a_{g}) = \Theta_{G/ G'}(a_{g G'}) \in S(G / G')$. 
\end{proof}

\clearpage

\begin{center}
Acknowledgements
\end{center}

We would like to thank Christopher Pinner for referring us to some references. 
One of the references eliminates the need to write a proof of Corollary~$\ref{cor:6}$. 

\bibliography{reference}

\begin{thebibliography}{1}

\bibitem{MR3882290}
Dilum De~Silva, Michael~J. Mossinghoff, Vincent Pigno, and Christopher Pinner.
\newblock The {L}ind-{L}ehmer constant for certain {$p$}-groups.
\newblock {\em Math. Comp.}, 88(316):949--972, 2019.

\bibitem{Frobenius1968gruppen}
Ferdinand~Georg Frobenius.
\newblock \"{U}ber die {P}rimfactoren der {G}ruppendeterminante.
\newblock {\em Sitzungsberichte der K\"{o}niglich Preu{\ss}ischen Akademie der
  Wissenschaften zu Berlin}, pages 1343--1382, 1896.
\newblock Reprinted in {\it Gesammelte Abhandlungen, Band III}. Springer-Verlag
  Berlin Heidelberg, New York, 1968, pages 38--77.

\bibitem{MR624127}
H.~Turner Laquer.
\newblock Values of circulants with integer entries.
\newblock In {\em A collection of manuscripts related to the {F}ibonacci
  sequence}, pages 212--217. Fibonacci Assoc., Santa Clara, Calif., 1980.

\bibitem{MR590367}
Michael~K. Mahoney and Morris Newman.
\newblock Determinants of abelian group matrices.
\newblock {\em Linear and Multilinear Algebra}, 9(2):121--132, 1980.

\bibitem{MR550657}
Morris Newman.
\newblock On a problem suggested by {O}lga {T}aussky-{T}odd.
\newblock {\em Illinois J. Math.}, 24(1):156--158, 1980.

\bibitem{MR3522826}
Vincent Pigno, Chris Pinner, and Wasin Vipismakul.
\newblock The {L}ind-{L}ehmer constant for {$\Bbb Z_m\times\Bbb Z^n_p$}.
\newblock {\em Integers}, 16:Paper No. A46, 12, 2016.

\bibitem{https://doi.org/10.48550/arxiv.2203.14420}
Naoya Yamaguchi and Yuka Yamaguchi.
\newblock Generalized {D}edekind's theorem and its application to integer group
  determinants.
\newblock To appear in {\it J. Math. Soc. Japan}.

\bibitem{MR4526227}
Naoya Yamaguchi and Yuka Yamaguchi.
\newblock Remark on {L}aquer's theorem for circulant determinants.
\newblock {\em Int. J. Group Theory}, 12(4):265--269, 2023.

\end{thebibliography}
\bibliographystyle{plain}

\medskip
\begin{flushleft}
Faculty of Education, 
University of Miyazaki, 
1-1 Gakuen Kibanadai-nishi, 
Miyazaki 889-2192, 
Japan \\ 
{\it Email address}, Yuka Yamaguchi: y-yamaguchi@cc.miyazaki-u.ac.jp \\ 
{\it Email address}, Naoya Yamaguchi: n-yamaguchi@cc.miyazaki-u.ac.jp 
\end{flushleft}

\end{document}